\documentclass[11pt]{amsart}
\title[Twisted Lefschetz numbers of infra-solvmanifolds]
{Twisted Lefschetz numbers of infra-solvmanifolds and  algebraic groups }

\author{Hisashi Kasuya}

\usepackage{amssymb}
\usepackage{amsmath}
\usepackage{amscd}
\usepackage{amstext}
\usepackage{amsfonts}
\usepackage[all]{xy}

\theoremstyle{plain}

\theoremstyle{plain}

\theoremstyle{plain}

\theoremstyle{plain}
\newtheorem{theorem}{Theorem}[section] 
\theoremstyle{remark}
\newtheorem{remark}{Remark}
\theoremstyle{Main result}
\newtheorem{main result}{Main result}
\theoremstyle{lemma}
\newtheorem{lemma}[theorem]{Lemma}
\theoremstyle{definition}
\newtheorem{definition}[theorem]{Definition}
\theoremstyle{proposition}
\newtheorem{proposition}[theorem]{Proposition}
\theoremstyle{corollary}

\theoremstyle{remark}
\newtheorem{example}{Example}

\address[Hisashi Kasuya]{Department of Mathematics, Graduate School of Science, Osaka University, Osaka,  Japan.}
\email{kasuya@math.sci.osaka-u.ac.jp}

\keywords{Topological fixed point theory, cohomology of algebraic group, twisted lefschetz number,  infra-solvmanifold}
\subjclass[2010]{22E25;	20G20 ;  54H25; 55M20  }

\newcommand{\C}{\mathbb{C}}
\newcommand{\R}{\mathbb{R}}

\newcommand{\Z}{\mathbb{Z}}
\newcommand{\g}{\frak{g}}

\begin{document} 

\maketitle
\begin{abstract}
Twisted Lefschetz numbers are extensions of the ordinary Lefschetz numbers for cohomologies with values in flat bundles.
As a generalization of  linearization formula for the  ordinary Lefschetz number of a diffeomorphism of a  nilmanifold, we show that a twisted Lefschetz number of any diffeomorphism  of any infra-solvmanifold is equal to the determinant ${\rm det}(I-A)$ for some matrix $A$ by using the cohomology of algebraic groups.
\end{abstract}
\section{Introduction}

Let $M$ be a compact manifold and $E$ be a flat vector bundle over $M$.
We consider the de Rham cohomology $H^{\ast}(M,E)$ with values in  $E$.
Let $\varphi:M\to M$ be a smooth map and a morphism $\Xi: \varphi^{\ast}E\to E$ of flat bundles  where $\varphi^{\ast}E $ is the pull-back of $E$.
Consider the induced map $\Xi\circ H^{\ast}(\varphi): H^{\ast}(M,E)\to H^{\ast}(M,E)$.
We define the twisted Lefschetz number
\[L(\varphi,E, \Xi)=\sum_{i}(-1)^{i}{\rm trace} (\Xi\circ H^{\ast}(\varphi))_{\vert H^{i}(M,E)}.\]
As for the ordinary Lefschetz number, we can show the Lefschetz fixed point formula for the twisted Lefschetz number by the Atiyah-Bott theorem (\cite{AB}).
If $\varphi $ is a diffeomorphism so that the graph of $\varphi$ is transversal to the diagonal in $M\times M$,
then we have
\[L(\varphi,E, \Xi)=\sum_{\varphi(x)=x}{\rm sign}({\rm det} (I-d\varphi_{x}))\cdot {\rm trace}(\Xi_{x}).
\]

In general, computations of the twisted Lefschetz numbers are harder than computations of  usual Lefschetz numbers.
In this paper we give a simple formula for  twisted Lefschetz numbers of diffeomorphisms of solvmanifolds and infra-solvmanifolds  which is a generalization of the  "linearization formula" for  ordinary Lefschetz numbers of maps of nilmanifolds.

In this paper solvmanifolds (resp. nilmanifolds) are compact homogeneous spaces of connected solvable (resp. nilpotent) Lie groups.
We say that a solvmanifold (resp. nilmanifold) $M$ has a discrete presentation $(G,\Gamma)$ if there exists a simply connected solvable (resp. nilpotent) Lie group $G$ with a cocompact discrete subgroup $\Gamma$ so that $M$ is diffeomorphic to $G/\Gamma$.
It is known that every nilmanifold $M$ has a discrete presentation $(G,\Gamma)$ (see \cite{Mal}).

Consider a nilmanifold $G/\Gamma$ for a discrete presentation $(G,\Gamma)$.
For a smooth map $\varphi: G/\Gamma \to G/\Gamma$, since the fundamental group of $G/\Gamma$ is $\Gamma$,
 we have the  map $f: \Gamma\to \Gamma$ induced by $\varphi$. 
It is known that any homomorphism between cocompact discrete subgroups of simply connected nilpotent Lie groups can be extended to a Lie group homomorphism.
Hence we have an extension $T:G\to G$ of $f$.
In \cite{Nom}, Nomizu showed that the de Rham cohomology of  a nilmanifold $G/\Gamma$ for a discrete presentation $(G,\Gamma)$ is isomorphic to the Lie algebra cohomology of the Lie algebra $\g$ of $G$.
Hence, considering the induced homomorphism $T_{\ast}:\g\to \g$, taking a matrix presentation $A$ of  $T_{\ast}:\g\to \g$, we can get the "linearization formula"
\[L(\varphi)={\rm det}(I-A)
\]
where $L(\varphi)$ is the ordinary Lefschetz number.

We are interested in getting a "linearization formula" for solvmanifolds.
But there are many difficulties:
\begin{itemize}
\item In general, a solvmanifold $M$ does not have a discrete presentation $(G,\Gamma)$.
\item  In general, for a simply connected solvable Lie group with a cocompact discrete subgroup $\Gamma$, a homomorphism $\Gamma\to \Gamma$ may not extend to a Lie group homomorphism $G\to G$.
\item  In general, the de Rham cohomology of a solvmanifold $G/\Gamma$ for a discrete presentation $(G,\Gamma)$ is not isomorphic to the Lie algebra cohomology of the Lie algebra $\g$ of $G$.

\end{itemize}
Assuming some conditions, we can get a "linearization formula" (see for examples \cite{Mc} and \cite{JM}).
However on general  solvmanifolds, it seems to be difficult to get  a "linearization formula".

In this paper we consider the more general setting.
 An infra-solvmanifold is a manifold of the form $G/\Delta$, where $G$ is a simply connected solvable Lie group, and $\Delta$ is a torsion-free subgroup of ${\rm Aut}(G)\ltimes G$ such that the closure of $h(\Delta)$ in ${\rm Aut}(G)$ is compact where $h:{\rm Aut}(G)\ltimes G\to {\rm Aut}(G)$ is the projection.
An infra-solvmanifold is a  generalization of a solvmanifold.
See \cite{B} and \cite{FJ} for properties of infra-solvmanifolds.
The purpose of this paper is to give a "linearization formula" for any diffeomorphism of any infra-solvmanifold by using the twisted  Lefschetz number and algebraic groups.

Let  $M$ be an infra-solvmanifold with the fundamental  group $\Gamma$ and  $\varphi: M\to M$  a diffeomorphism.
Denote by $f:\Gamma\to \Gamma$ the homomorphism induced by $\varphi$.
Then, by methods of algebraic groups given in \cite{R, B}, we can say that $f:\Gamma\to \Gamma$ induces an explicit isomorphism on a certain Lie algebra ${\frak u}_{\Gamma}$.
For a matrix presentation $A$ of such isomorphism, we will prove that by taking some $E$ and $\Xi$, the formula 
\[L(\varphi,E, \Xi)={\rm det}(I-A)
\]
holds.

\section{Algebraic groups and representations}
An  algebraic group $\mathcal G$ is a Zariski-closed subgroup of $GL_{n}(\C)$.
 We denote  by ${\mathcal  U}({\mathcal G})$  the unipotent radical of  $\mathcal G$.
If ${\mathcal  U}({\mathcal G})$ is trivial, an algebraic group $\mathcal G$ is called reductive.

Denote by  ${\rm Rep}(\mathcal G)$  the category of finite-dimensional rational representations $V_{\rho}$ of $\mathcal G$.
$V_{\rho}\in {\rm Rep}(\mathcal G)$ says $V_{\rho}$ is a $\C$-vector space with a rational representation $\rho:{\mathcal G}\to GL(V_{\rho})$.
For $V_{\rho}\in {\rm Rep}(\mathcal G)$, we denote by $\overline{V_{\rho}}$ the $\C$-vector space $V_{\rho}$ with the trivial representation.

\section{Representations of reductive algebraic groups}
Let $\mathcal G$ be a reductive  algebraic group and $V_{\rho}\in {\rm Rep}(\mathcal G)$.
We assume that we have an algebraic group isomorphism $F: \mathcal G\to \mathcal G$  and a linear isomorphism $\Phi: V_{\rho}\to V_{\rho}$ satisfying 
\[\rho(g)\circ\Phi=\Phi \circ \rho(F(g))
\]
for every $g\in \mathcal G$.

We take a finite set $\Lambda_{\rho}=\{V_{\alpha}\}$ of irreducible representations of $\mathcal G$ such that 
we have an irreducible decomposition $V_{\rho}=\bigoplus_{\Lambda_{\rho}} E_{\alpha}$ and $E_{\alpha}$ is a non-trivial irreducible component of $V_{\rho}$ corresponding to an irreducible $V_{\alpha}\in  {\rm Rep}(\mathcal G)$.
Denote $E_{\alpha\circ F}=\Phi(E_{\alpha})$. Then we have the irreducible decomposition $V_{\rho}=\bigoplus_{\Lambda_{\rho}} E_{\alpha\circ F}$.
Consider $V_{\alpha\circ F}\in  {\rm Rep}(\mathcal G)$ as the vector space $V_{\alpha}$ with the representation $\alpha\circ F$.
Then, $E_{\alpha\circ F}$ is a non-trivial irreducible component of $V_{\rho}$ corresponding to an irreducible $V_{\alpha\circ F}\in  {\rm Rep}(\mathcal G)$.
Comparing two irreducible decomposition $V_{\rho}=\bigoplus_{\Lambda_{\rho}} E_{\alpha}=\bigoplus_{\Lambda_{\rho}} E_{\alpha\circ F}$, for each $\alpha\in \Lambda_{\rho}$, we have a unique $\beta\in  \Lambda_{\rho}$ such that we have an isomorphism $\psi_{\alpha}:V_{\alpha\circ F}\cong V_{\beta}$.

Let us consider the linear map $\Phi_{\alpha}: V_{\rho}\otimes V_{\alpha}^{\ast}\otimes \overline{V_{\alpha}}\to V_{\rho}\otimes V_{\alpha\circ F}^{\ast}\otimes \overline{V_{\alpha\circ F}}$ defined by 
$\Phi_{\alpha}(f_{\alpha}\otimes v_{\alpha})=\Phi\circ f_{\alpha}\otimes v_{\alpha}$ for $f_{\alpha}\in V_{\rho}\otimes V_{\alpha}^{\ast}$, $v_{\alpha}\in \overline{V_{\alpha}}$ (we should notice $\overline{V_{\alpha}}=\overline{V_{\alpha\circ F}}$).
Then this linear map is $\mathcal G$-equivariant and hence we can restrict $\Phi_{\alpha}: (V_{\rho}\otimes V_{\alpha}^{\ast})^{\mathcal G}\otimes \overline{V_{\alpha}}\to (V_{\rho}\otimes V_{\alpha\circ F}^{\ast})^{\mathcal G}\otimes \overline{V_{\alpha\circ F}}$.
Consider the isomorphism $(\psi_{\alpha}^{-1})^{\ast}\otimes \psi_{\alpha}: V_{\alpha\circ F}^{\ast}\otimes \overline{V_{\alpha\circ F}}\to  V_{\beta}^{\ast}\otimes \overline{V_{\beta}}$.
Then we have the linear map $((\psi_{\alpha}^{-1})^{\ast}\otimes \psi_{\alpha})\circ \Phi_{\alpha}: (V_{\rho}\otimes V_{\alpha}^{\ast})^{\mathcal G}\otimes \overline{V_{\alpha}}\to (V_{\rho}\otimes  V_{\beta}^{\ast})^{\mathcal G}\otimes \overline{V_{\beta}}$.

We consider $V_{\phi}\in  {\rm Rep}(\mathcal G)$ as $V_{\phi}=\bigoplus_{\Lambda_{\rho}}(V_{\rho}\otimes V_{\alpha}^{\ast})^{\mathcal G}\otimes \overline{V_{\alpha}}$.
Then we have the linear map $\bigoplus_{\Lambda_{\rho}}((\psi_{\alpha}^{-1})^{\ast}\otimes \psi_{\alpha})\circ \Phi_{\alpha}: (V_{\rho}\otimes V_{\phi})^{\mathcal G} \to (V_{\rho}\otimes V_{\phi})^{\mathcal G}$.

\begin{lemma}\label{Id}
There exists an linear isomorphism $(V_{\rho}\otimes V_{\phi})^{\mathcal G}\cong  \overline{V_{\rho}}$ which identifies 
 the linear map $\bigoplus_{\Lambda_{\rho}}((\psi_{\alpha}^{-1})^{\ast}\otimes \psi_{\alpha})\circ \Phi_{\alpha}: (V_{\rho}\otimes V_{\phi})^{\mathcal G} \to (V_{\rho}\otimes V_{\phi})^{\mathcal G}$ with 
 $\Phi:  \overline{V_{\rho}}\to  \overline{V_{\rho}}$.
\end{lemma}
\begin{proof}
By \cite[Theorem 27.3.6]{TY}, we have the isomorphism
 \[(V_{\rho}\otimes V_{\alpha}^{\ast})^{\mathcal G}\otimes \overline{V_{\alpha}}\ni f_{\alpha}\otimes v_{\alpha}\mapsto f(v_{\alpha})\in \overline{E_{\alpha}}.\]
Thus we have the isomorphism $(V_{\rho}\otimes V_{\phi})^{\mathcal G}\cong \bigoplus_{\Lambda_{\rho}}\overline{E_{\alpha}}=\overline{V_{\rho}}$.
Since we have
\[((\psi_{\alpha}^{-1})^{\ast}\otimes \psi_{\alpha})\circ \Phi_{\alpha} (f_{\alpha}\otimes v_{\alpha})=\Phi\circ f_{\alpha}\circ \psi^{-1}_{\alpha}\otimes \psi_{\alpha}(v_{\alpha}),
\]
the lemma follows from this isomorphism.

\end{proof}

\section{Rational cohomology of algebraic groups}\label{RC}

Let $\mathcal G$ be an algebraic group.
In this section,  we denote ${\mathcal  U}={\mathcal  U}({\mathcal G})$
Let $F:{\mathcal G}\to  {\mathcal G}$ be an isomorphism of  algebraic groups.
Consider  the extension
\begin{equation}\label{exun}
\xymatrix{
1\ar[r]&{\mathcal U}\ar[r]&\mathcal G\ar[r]&\mathcal G/{\mathcal U}\ar[r]&1
}.
\end{equation}
It is known that this extension  splits (\cite{Mos1}).
The restriction $F_{\mathcal U}:{\mathcal U}\to {\mathcal U}$ is also an isomorphism and $F$ induces an isomorphism ${\bar F}:{\mathcal G}/{\mathcal U}\to {\mathcal G}/{\mathcal U}$.

For   $V_{\phi}\in {\rm Rep}(\mathcal G)$, we define the rational cohomology $H^{\ast}({\mathcal G},V_{\phi})={\rm Ext}^{\ast}_{\mathcal G}(\C,V_{\phi})$ as in \cite{Hoc} and \cite{Jan}.
Consider the induced map $H^{\ast}(F): H^{\ast}({\mathcal G} ,V_{\phi})\to H^{\ast}({\mathcal G}, V_{\phi\circ F})$.
Let $\Psi: V_{\phi\circ F}\to V_{\phi}$ be an isomorphism in ${\rm Rep}(\mathcal G)$.
We consider the linear map $\Psi\circ  H^{\ast}(F): H^{\ast}({\mathcal G} ,V_{\phi})\to  H^{\ast}({\mathcal G} ,V_{\phi})$.
It is known that we have a natural isomorphism
\[H^{\ast}({\mathcal G}, V_{\phi})
 \cong  H^{\ast}({\mathcal U}, V_{\phi})^{\mathcal G/{\mathcal U}}\cong H^{\ast}({\frak u}, V_{\phi})^{\mathcal G/{\mathcal U}}
\]  induced by the extension  (\ref{exun}) and the exponential map ${\frak u} \to {\mathcal U}$
(see \cite{Hoc},\cite{Haia})
where ${\frak u}$ is the Lie algebra of ${\mathcal U}$ and $H^{\ast}({\frak u}, V_{\phi})$ is the Lie algebra cohomology and we regard $H^{\ast}({\mathcal U}, V_{\phi}), H^{\ast}({\frak u}, V_{\phi})\in {\rm Rep}(\mathcal G/{\mathcal U})$ via the adjoint action.
This isomorphism identifies the linear map $\Psi\circ  H^{\ast}(F): H^{\ast}({\mathcal G} ,V_{\phi})\to  H^{\ast}({\mathcal G} ,V_{\phi})$ with 
the linear map $\Psi\circ  H^{\ast}(F_{{\mathcal U} \ast}): H^{\ast}({\frak u}, V_{\phi})^{\mathcal G/{\mathcal U}}\to  H^{\ast}({\frak u}, V_{\phi})^{\mathcal G/{\mathcal U}}$ where $F_{{\mathcal U} \ast}:{\frak u}\to {\frak u}$ is the derivation of $F_{\mathcal U}:{\mathcal U}\to {\mathcal U}$.

Denote by $H^{\ast}({\frak u})$ the Lie algebra cohomology with values in the $1$-dimensional trivial representation.
We apply the arguments in the last section to the objects:
\begin{itemize}
\item The reductive algebraic group $\mathcal G/{\mathcal U}$.
\item The representation $H^{\ast}({\frak u})\in {\rm Rep}(\mathcal G/{\mathcal U})$.
\item The  isomorphism ${\bar F}:{\mathcal G}/{\mathcal U}\to {\mathcal G}/{\mathcal U}$.
\item The linear isomorphism $H^{\ast}(F_{{\mathcal U} \ast}): H^{\ast}({\frak u})\to H^{\ast}({\frak u})$.
\end{itemize}
Then, taking $V_{\phi}\in {\rm Rep}({\mathcal G}/{\mathcal U})$ and $\Psi=\bigoplus_{\Lambda_{\rho}}((\psi_{\alpha}^{-1})^{\ast}\otimes \psi_{\alpha})$ as in the last section and extending $V_{\phi}\in {\rm Rep}({\mathcal G})$ by  the extension  (\ref{exun}), the isomorphism
\[H^{\ast}({\frak u}, V_{\phi})^{\mathcal G/{\mathcal U}}=(H^{\ast}({\frak u})\otimes V_{\phi})^{\mathcal G/{\mathcal U}}\cong \overline{H^{\ast}({\frak u}})
\]
 in Lemma \ref{Id}  identifies the linear map $\Psi\circ  H^{\ast}(F_{{\mathcal U} \ast}): H^{\ast}({\frak u}, V_{\phi})^{\mathcal G/{\mathcal U}}\to  H^{\ast}({\frak u}, V_{\phi})^{\mathcal G/{\mathcal U}}$ with the linear map
 $H^{\ast}(F_{{\mathcal U} \ast}): \overline{H^{\ast}({\frak u})}\to \overline{H^{\ast}({\frak u})}$.
This implies the following:

\begin{proposition}\label{tratar1}
Let $A$ be a matrix presentation of the linear map
$F_{{\mathcal U} \ast}:\frak u\to \frak u$.
Then we have 
\[\sum_{i}(-1)^{i}{\rm trace} (\Psi\circ  H^{\ast}(F))_{\vert H^{i}({\mathcal G}, V_{\phi})}={\rm det}(I-A).\]

\end{proposition}
\section{Group cohomology of torsion-free virtually polycyclic groups}

\begin{lemma}{\rm(\cite[Lemma 4.36.]{R})}\label{hut}
Let $\Gamma$  be a torsion-free virtually polycyclic group.
  For a finite-dimensional  representation $\rho:\Gamma\to GL(V_{\rho})$ on a complex vector space $V_{\rho}$, denote by $\mathcal G$ the  Zariski-closure of $\rho(\Gamma)$  in $ GL(V_{\rho})$. Then we have $\dim \mathcal U(\mathcal G)\le {\rm rank}\, \Gamma$.
\end{lemma}

By this lemma, we consider the following definition.
\begin{definition}
Let $\Gamma$ be a torsion-free virtually polycyclic group. 
 For a real algebraic group $\mathcal G$, a  representation $\rho:\Gamma\to \mathcal G$   is called  {\em full} if the image $\rho(\Gamma)$ is Zariski-dense in $ \mathcal G$ and we have $\dim \mathcal U(\mathcal G)= {\rm rank}\, \Gamma$.
\end{definition}

We can compute the group cohomology of  torsion-free virtually polycyclic groups by using rational cohomology of algebraic groups.
Let  ${\rm Rep}(\Gamma)$ be the category of  representations $V_{\rho}$ of $\Gamma$.
For $V_{\rho}\in{\rm Rep}(\Gamma)$,  we define the group cohomology $H^{\ast}(\Gamma,V_{\rho})={\rm Ext}^{\ast}_{\Gamma}(\C,V_{\rho})$ as in \cite{Br}.
\begin{theorem}[\cite{KCT}]\label{poly1}
Let $\Gamma$ be a torsion-free virtually polycyclic group.
We suppose that  we have an algebraic group $\mathcal G$ and an injection $\Gamma\subset {\mathcal G}$ which is a full representation.
Let $V$ be a rational $\mathcal G$-module.
Then  the inclusion $\Gamma\subset {\mathcal G}$ induces a cohomology isomorphism
\[H^{\ast}({\mathcal G},V_{\rho})\cong H^{\ast}(\Gamma,V_{\rho}).
\]
\end{theorem}
This isomorphism was discovered by Baues for  the case of trivial coefficients see the sentence after \cite[Theorem 1.8]{B}.

For any torsion-free virtually polycyclic group, we have a good full representation.
\begin{definition}\label{a-d-2}
We call an  algebraic group  ${\mathcal A}_{\Gamma}$ an {\em algebraic hull} of $\Gamma$ if there exists an injective group homomorphism $\psi:\Gamma\to {\mathcal A}_{\Gamma}$
so that\begin{itemize}
\item  $\psi:\Gamma\to {\mathcal A}_{\Gamma}$ is a full representation.
\item   $Z_{{\mathcal A}_{\Gamma}}({\mathcal U}_{\Gamma})\subset {\mathcal U}_{\Gamma}$ where  $Z_{{\mathcal A}_{\Gamma}}({\mathcal U}_{\Gamma})$ is the centralizer of the unipotent radical  ${\mathcal U}_{\Gamma}$ of ${\mathcal A}_{\Gamma}$. 

\end{itemize} 
\end{definition}
\begin{theorem}\label{a-t-1}{\rm (\cite[Theorem A.1]{B})}
There exists an algebraic hull of $\Gamma$ and an algebraic hull of $\Gamma$ is unique up to algebraic group isomorphism.
\end{theorem}

Let $\Gamma$ be a torsion-free virtually polycyclic group and $f:\Gamma\to \Gamma$ an isomorphism.
Take an algebraic hull $\psi:\Gamma\to {\mathcal A}_{\Gamma}$.
Then, $f$ extends uniquely to an isomorphism $F: {\mathcal A}_{\Gamma}\to {\mathcal A}_{\Gamma}$ by \cite[Lemma 4.41]{R}.
In fact, we obtain $F$ by the following way.

Define $\Phi_{f}: \Gamma\to  {\mathcal A}_{\Gamma}\times  {\mathcal A}_{\Gamma}$ as $\Phi_{f}(\gamma)=(\psi(\gamma), \psi\circ f(\gamma))$.
Let ${\mathcal G}_{f}$ be the real Zariski-closure of  $\Phi_{f}(\Gamma)$ in $ {\mathcal A}_{\Gamma}\times  {\mathcal A}_{\Gamma}$.
For the projections $p_{1}:{\mathcal A}_{\Gamma}\times  {\mathcal A}_{\Gamma}\ni (a,b)\mapsto (a)\in {\mathcal A}_{\Gamma}$ and $p_{2}:{\mathcal A}_{\Gamma}\times  {\mathcal A}_{\Gamma}\ni (a,b)\mapsto (b)\in {\mathcal A}_{\Gamma}$, we consider the restriction $\pi_{i}:{\mathcal G}_{f}\to {\mathcal A}_{\Gamma}$ of $p_{i}$ on ${\mathcal G}_{f}$ for $i=1,2$.
Then $\pi_{1}$ is an isomorphism and desired $F$ is $\pi_{2}\circ \pi_{1}^{-1}$.

For the isomorphism $F: {\mathcal A}_{\Gamma}\to {\mathcal A}_{\Gamma}$, by Theorem \ref{poly1},
for $V_{\phi}\in {\rm Rep}({\mathcal A}_{\Gamma})$, 
the induced map on the cohomology
\[H^{\ast}(f):H^{\ast}(\Gamma,V_{\phi})\to H^{\ast}(\Gamma,V_{\phi\circ f})
\]
is identified with the map
\[H^{\ast}(F):H^{\ast}({\mathcal A}_{\Gamma},V_{\phi})\to H^{\ast}({\mathcal A}_{\Gamma},V_{\phi\circ F}).
\]

We apply the arguments in the last section to the algebraic group ${\mathcal A}_{\Gamma}$.
Take  $V_{\phi}\in {\rm Rep}({\mathcal A}_{\Gamma}/{\mathcal U}_{\Gamma})$ and $\Psi$ for Proposition \ref{tratar1}.
The isomorphisms 
\[H^{\ast}(\Gamma,V_{\phi})\cong H^{\ast}({\mathcal A}_{\Gamma},V_{\phi})\cong H^{\ast}({\frak u}_{\Gamma}, V_{\phi})^{\mathcal G/{\mathcal U}_{\Gamma}}=(H^{\ast}({\frak u}_{\Gamma})\otimes V_{\phi})^{\mathcal G/{\mathcal U}_{\Gamma}}\cong \overline{H^{\ast}({\frak u}_{\Gamma}})
\] 
identify the linear map $\Psi \circ H^{\ast}(f):H^{\ast}(\Gamma,V_{\phi})\to H^{\ast}(\Gamma,V_{\phi}) $ with  $H^{\ast}(F_{\mathcal U_{\Gamma}\ast }):  \overline{H^{\ast}({\frak u}_{\Gamma})}\to  \overline{H^{\ast}({\frak u}_{\Gamma})}$ where 
$\frak u_{\Gamma}$ is the Lie algebra of ${\mathcal U}_{\Gamma}$.
\begin{proposition}\label{ggg}
Let $A$ be a matrix presentation of the linear map
$F_{\mathcal U\ast}:\frak u_{\Gamma}\to \frak u_{\Gamma}$.
Then we have 
\[\sum_{i}(-1)^{i}{\rm trace} (\Psi \circ H^{\ast}(f))_{\vert H^{i}(\Gamma, V_{\phi})}={\rm det}(I-A).\]

\end{proposition}
\section{Main result}\label{TWI}
We give a "linearization formula" for any diffeomorphism of any infra-solvmanifold by using the twisted  Lefschetz number and algebraic groups.
It is known  that the fundamental group of every infra-solvmanifold is a torsion-free virtually polycyclic group (see \cite{FJ}).

Let  $M$ be an infra-solvmanifold with the fundamental  group $\Gamma$ and  $\varphi: M\to M$  a diffeomorphism.
Denote by $f:\Gamma\to \Gamma$ the homomorphism induced by $\varphi$.
For the universal covering $\tilde M$, we write $M=\tilde M/\Gamma$.
Then, $\varphi$ is described as a diffeomorphim $\tilde\varphi: \tilde M\to \tilde M$ so that for $x\in \tilde M$ and $g\in \Gamma$, $ \tilde\varphi(gx)=f(g)\tilde\varphi(x)$.
For $V_{\phi}\in {\rm Rep}(\Gamma)$, we define the flat bundle $E_{\phi}=(\tilde M\times V_{\phi})/\Gamma$.
Then $\varphi^{\ast}E_{\phi}=E_{\phi\circ f}$.
By the standard theory (see \cite[Lemma 7.4]{R} for instance), 
 the induced map $H^{\ast}(\varphi):H^{\ast}(M, E_{\phi})\to H^{\ast}(M, \varphi^{\ast}E_{\phi})$ is identified 
 with $H^{\ast}(f):H^{\ast}(\Gamma, V_{\phi})\to H^{\ast}(\Gamma, V_{\phi\circ f})$.
 Take  $V_{\phi}\in {\rm Rep}({\mathcal A}_{\Gamma}/{\mathcal U}_{\Gamma})$ and an isomorphism  $\Psi: V_{\phi\circ F}\to V_{\phi}$  as in the last section.
 Considering $\Psi: V_{\phi\circ f}\to V_{\phi}$ as an isomorphism in ${\rm Rep}(\Gamma)$,  $\Psi$ corresponds to an  isomorphism  $\Xi:\varphi^{\ast}E_{\phi}\to E_{\phi}$ of flat bundles.

Thus, by Proposition \ref{ggg}, we obtain the following result.
\begin{theorem}\label{MTTT}

Let $A$ be a matrix presentation of  $F_{\mathcal U_{\Gamma}\ast}:\frak u\to \frak u$.
Then we have 
\[L(\varphi, E_{\phi}, \Xi)={\rm det}(I-A).
\]
\end{theorem}

\begin{remark}
We have a good presentation of the universal covering $\tilde M$ of any infra-solvmanifold $M$.
For a  torsion-free virtually polycyclic group $\Gamma$ with an algebraic hull $\psi:\Gamma\to {\mathcal A}_{\Gamma}$,
since the extension (\ref{exun}) splits, we have a splitting ${\mathcal A}_{\Gamma}={\mathcal A}_{\Gamma}/{\mathcal U}_{\Gamma}\ltimes {\mathcal U}_{\Gamma}$.
We can take ${\mathcal A}_{\Gamma}$ defined over $\R$ and $\psi^{\prime}(\Gamma)$ contained in the $\R$-rational points ${\mathcal A}_{\Gamma}(\R)$ of   ${\mathcal A}_{\Gamma}$.
By an injection  $\psi:\Gamma\to {\mathcal A}_{\Gamma}(\R)$, we have the $\Gamma$-action on  ${\mathcal U}_{\Gamma}(\R)$ by the splitting.
In \cite{B},  Baues proved that the quotient space ${\mathcal U}_{\Gamma}(\R)/\Gamma$ is a compact aspherical manifold with the universal  covering ${\mathcal U}_{\Gamma}(\R)$ and every infra-solvmanifold $M$ with the fundamental group $\Gamma$ is diffeomorphic to  ${\mathcal U}_{\Gamma}(\R)/\Gamma$.
Thus, we can take $\tilde M={\mathcal U}_{\Gamma}(\R)$.
We notice that as a Lie group, ${\mathcal U}_{\Gamma}(\R)$ is simply connected nilpotent Lie group and by the exponential map,  ${\mathcal U}_{\Gamma}(\R)$ is diffeomorphic to the Euclidean space of dimension $\dim {\mathcal U}_{\Gamma}$.

\end{remark}

\begin{remark}
If $M$ is a nilmanifold, then $\Gamma$ is a nilpotent and we have ${\mathcal A}_{\Gamma}={\mathcal U}_{\Gamma}$ (see \cite{R}) and the above quotient space ${\mathcal U}_{\Gamma}(\R)/\Gamma$ is a quotient of a simply connected nilpotent Lie group ${\mathcal U}_{\Gamma}(\R)$ by a discrete subgroup $\Gamma$.
In this case,  the flat bundle $ E_{\phi}$ is trivial and so Theorem \ref{MTTT} is the usual linearization formula  for a nilmanifold as we explained in Introduction.
\end{remark}

\begin{example}
Let $G=\R\ltimes_{\lambda}\R^{2}$ such that $\lambda(t)=\left(\begin{array}{cc} \cos \pi t & -\sin \pi t \\ \sin \pi t &\cos \pi t\end{array}\right)$ for $t\in \R$.
Then the subgroup $\Gamma=\Z\ltimes_{\lambda}\Z^{2}$ is a cocompact discrete subgroup in $G$.
Consider the solvmanifold $M=G/\Gamma$.
We have the explicit  algebraic hull $\psi:\Gamma\to {\mathcal A}_{\Gamma}$ where ${\mathcal A}_{\Gamma}=\langle \tau \rangle \ltimes \C^{3}$ such that $\tau=\left(\begin{array}{ccc} 1 & 0&0 \\ 
0 &-1&0\\
0&0&-1\end{array}\right)$ and $\psi(t,u,v)=(\tau^{t},t,u,v)$ for $(t,u,v)\in \Gamma$.
We consider the automorphism $f:\Gamma\to \Gamma$ defined   by the matrix $A=\left(\begin{array}{ccc} -1 & 0&0 \\ 
0 &a&b\\
0&c&d\end{array}\right)$ such that $\left(\begin{array}{cc} a &b \\ c &d\end{array}\right)\in SL(2,\Z)$. 
This $f$ can not be extended to an automorphism of the Lie group $G$.
We can extend $f$ to the automorphism $F: {\mathcal A}_{\Gamma}\to  {\mathcal A}_{\Gamma}$ such that $F(\tau)=\tau$ and $F$ on ${\mathcal U}_{\Gamma}=\C^{3}$ is  defined by  the matrix $A$.
We can compute 
\[H^{i}(M,\C)=H^{i}(\Gamma,\C)=H^{i}(\C^{3})^{\langle \tau \rangle}=\left\{\begin{array}{cc} \C &(0\le i\le 3) \\ 0 & ({\rm other }\,\, i )\end{array}\right. 
\]
and we have $H^{\ast}(f)_{\vert H^{i}(\Gamma,\C)}=(-1)^{i}$.
Consider a diffeomorphism $\varphi:M\to M$ inducing $f$ on the fundamental group $\Gamma$ (e.g. $\varphi$ is the map on $M=\R^{3}/\Gamma$ defined by  the matrix $A$ where we consider $G=\R^{3}$ as a smooth manifold).
Then the ordinary Lefschetz number is $L(\varphi)=4$.
In contrast, for the above $E_{\phi}$ and $\Xi$, by Theorem  \ref{MTTT}, we have
\[L(\varphi, E_{\phi}, \Xi)={\rm det}(I-A)=4-2(a+d).\]

\end{example}

\section{Application}
We consider a solvmanifold $G/\Gamma$ for a discrete presentation $(G,\Gamma)$ and  a  diffeomorphism $\varphi: G/\Gamma\to G/\Gamma$.
Let $\g$ be the Lie algebra of $G$.
We suppose that the induced map $f:\Gamma\to \Gamma $ can be extended to a Lie group homomorphism $T:G\to G$.
Consider the derivation $T_{\ast}:\g\to \g$.
If an isomorphism $H^{\ast}(G/\Gamma,\R)\cong H^{\ast}(\g,\R)$ holds, then we have the equation
\[L(\varphi)={\rm det}(I-B)
\]
where $B$ is a matrix presentation of  $T_{\ast}:\g\to \g$.
But in general an isomorphism $H^{\ast}(G/\Gamma,\R)\cong H^{\ast}(\g,\R)$ does not hold.
Thus the equation $L(\varphi)={\rm det}(I-B)$ does not hold.
Even if an isomorphism $H^{\ast}(G/\Gamma,\R)\cong H^{\ast}(\g)$ does not hold, it is expected that the number ${\rm det}(I-B)$ gives some information for fixed points.
As an application of Theorem \ref{MTTT}, we will answer such expectation.

A solvmanifold $G/\Gamma$ is an infra-solvmanifold with the fundamental group $\Gamma$.
For the same notation as in Theorem \ref{MTTT}, we have:
\begin{theorem}

\[L(\varphi, E_{\phi}, \Xi)={\rm det}(I-B).
\]
\end{theorem}

\begin{proof}
Let $G$ be a simply connected solvable Lie group.
For a finite dimensional representation $\rho: G\to  GL(V_{\rho})$, taking the Zariski-closure $\mathcal G$ of $\rho(G)$,  f
we have $\dim {\mathcal U}({\mathcal G})\le \dim G$  (\cite[Lemma 4.36.]{R}).  
For an algebraic group $\mathcal G$ and a representation $\rho: G\to \mathcal G$,
we say that $\rho$ is a full representation if  the image $\rho(G)$ is Zariski-dense in $ \mathcal G$ and we have $\dim \mathcal U= \dim  G$.

We call an  algebraic group  ${\mathcal A}_{G}$ an algebraic hull of $G$ if there exists an injective Lie  group homomorphism $\psi^{\prime}:G\to {\mathcal A}_{G}$
so that:\begin{itemize}
\item $\psi^{\prime}:G\to {\mathcal A}_{G}$ is a full representation.
\item  $Z_{{\mathcal A}_{G}}({\mathcal U}_{G})\subset {\mathcal U}_{G}$ where ${\mathcal U}_{G}$ is the unipotent radical of ${\mathcal A}_{G}$. 
\end{itemize}
As similar to virtually polycyclic case, 
there exists an algebraic hull ${\mathcal A}_{G}$ of $G$ and an algebraic hull of $G$ is unique up to algebraic group isomorphism (see \cite[Proposition 4.40, Lemma 4.41]{R}).
Moreover, we can take ${\mathcal A}_{G}$ defined over $\R$ and $\psi^{\prime}(G)$ contained in the $\R$-rational points ${\mathcal A}_{G}(\R)$ of   ${\mathcal A}_{G}$.

Take a real algebraic hull $\psi^{\prime}:G\to {\mathcal A}_{G}$ of $G$ such that ${\mathcal A}_{G}$  is defined over $\R$ and $\psi^{\prime}(G)\subset {\mathcal A}_{G}(\R)$.
Then, by the similar arguments on virtually polycyclic groups, for an isomorphism $T:G\to G$  of Lie groups, 
$T$ extends uniquely to an isomorphism $F: {\mathcal A}_{G}\to {\mathcal A}_{G}$.
It is known that the composition $Q=q\circ \psi^{\prime}: G\to {\mathcal U}_{G}(\R)$ as a smooth map is a diffeomorphism (see \cite[Proposition 2.3]{B}).
Hence, for the derivation $F_{\mathcal U_{G} \ast}:{\mathcal U}_{G}(\R)\to {\mathcal U}_{G}(\R)$, we have $F_{\ast}=Q_{\ast} T_{\ast} Q_{\ast}^{-1}$.
This implies that for matrix presentations $A$ and $B$ of $F_{\ast}$ and $T_{\ast}$ respectively,
we have \[{\rm det}(I-A)={\rm det}(I-B).\]

We suppose that $G$ has a cocompact discrete subgroup $\Gamma$ and $T:G\to G$ is an extension of  $f:\Gamma\to \Gamma $. 
It is known that the  Zariski-closure of $\psi^{\prime}(\Gamma)$ in ${\mathcal A}_{G}$ is the  algebraic hull ${\mathcal A}_{\Gamma}$ of a torsion-free polycyclic group $\Gamma$ and $ {\mathcal U}_{\Gamma}= {\mathcal U}_{G}$ (see \cite[Proof of Theorem 4.34]{R}).
Thus we can say that the restriction $F:{\mathcal A}_{\Gamma}\to {\mathcal A}_{\Gamma}$ is a unique extension of  $f:\Gamma\to \Gamma $.

Now we assume that $f:\Gamma\to \Gamma $ is induced by a  diffeomorphism $\varphi: G/\Gamma\to G/\Gamma$.
Then, by Theorem \ref{MTTT}, we have 
 \[L(\varphi, E_{\phi}, \Xi)={\rm det}(I-A).
\] 
Hence the theorem follows.

\end{proof}

\end{document}